\documentclass[10pt, reqno]{amsart}
\usepackage{amsthm, amsmath, amsfonts, amssymb, enumerate, textcmds, tensor, enumitem, mathdots}
\usepackage{pst-node}
\usepackage{tikz-cd} 
\usepackage{appendix}
\usepackage{graphicx,tikz}
\usepackage{nicefrac}
\usepackage{faktor}

\usepackage[style=alphabetic]{biblatex}
\usepackage{cleveref}

\usepackage{physics}
\usepackage{enumerate}
\usepackage{pinlabel}
\usepackage[colorinlistoftodos]{todonotes}
\newcommand{\R}{\mathbb R}

\newtheorem{theorem}{Theorem}[section]
\newtheorem{lemma}[theorem]{Lemma}

\newtheorem{prop}[theorem]{Proposition}

\newtheorem*{theorem estimates}{Theorem A}
\newtheorem*{theorema}{Theorem A}
\newtheorem*{theoremb}{Theorem B}

\DeclareMathOperator{\C}{\mathbb{C}}

\DeclareMathOperator{\CP}{\mathbb{C}\mathbb{P}}

\DeclareMathOperator{\PSL}{\mathrm{PSL}}

\DeclareMathOperator{\inners}{\langle \cdot, \cdot \rangle}

\newcommand{\ccpair}{(c_1, \overline{c_2})}
\DeclareMathOperator{\fm}{\vb*{\overline {f_{--}}}}
\DeclareMathOperator{\f+}{\vb*{f_+}}
\newcommand{\hpair}{\vb*h_{(c_1, \overline{c_2})}}

\usepackage{physics}
\usepackage{soul}

\usepackage{scalerel}

\usepackage{stmaryrd}

\usepackage{mathrsfs}
\usepackage{soul}

\theoremstyle{definition}

\theoremstyle{remark}
\newtheorem{remark}[theorem]{Remark}

\begin{document}
\title[Holomorphic dependence for the Beltrami equation in Sobolev spaces]{Holomorphic dependence for the Beltrami equation in Sobolev spaces}

\author[Christian El Emam]{Christian El Emam}
\address{Christian El Emam: University of Torino, Dipartimento di Matematica ``Giuseppe Peano", Via Carlo Alberto, 10, 10123 Torino, Italy.} \email{christian.elemam@unito.it}

\author[Nathaniel Sagman]{Nathaniel Sagman}
\address{Nathaniel Sagman: University of Luxembourg, 
Maison du Nombre,
6 Avenue de la Fonte,
L-4364 Esch-sur-Alzette, Luxembourg.} \email{nathaniel.sagman@uni.lu}

\begin{abstract}
We prove that, given a family of Beltrami forms on $\C$ with $L^\infty$ norm at most $\eta<1$ and that live in and vary holomorphically in the Sobolev space $W_{\textrm{loc}}^{l,\infty}(\Omega)$ of an open subset $\Omega\subset \C$, the canonical solutions to the Beltrami equation vary holomorphically in $W_{\textrm{loc}}^{l+1,p}(\Omega)$, for some $p=p(\eta)>2$. This extends a foundational result of Ahlfors and Bers (the case $l=0$). As an application, we deduce that Bers metrics depend holomorphically on their input data. 
\end{abstract}
\maketitle

\section{Introduction}

Throughout the paper, for an open subset $\Omega\subset \C$, $1\leq p\leq\infty$, and  $l\in\mathbb{Z}_{\geq 0}$, let $L^p(\Omega)$ and $W^{l,p}(\Omega)$ be the $L^p$ and Sobolev spaces respectively for maps from $\Omega$ to $\C$, and let $L_{\textrm{loc}}^p(\Omega)$ and $W_\textrm{loc}^{l,p}(\Omega)$ be the local versions, which we interpret as complex Fr{\'e}chet spaces. For $l=\infty,$ we use the notation $W_{\textrm{loc}}^{l,p}=C^\infty$. For $p=\infty$ and $\delta>0$, let $L_\delta^\infty(\Omega)$ be the open ball $\{\mu \in L^\infty(\Omega): ||\mu||_\infty<\delta\}$.

An orientation preserving homeomorphism $f$ between domains $\Omega, \Omega'\subset \C$ is \textbf{quasiconformal} if it lies in the Sobolev space $W_{\textrm{loc}}^{1,2}(\Omega)$ and if there exists a constant $0\leq k<1$ such that, for distributional derivatives $\partial_z f$ and $\partial_{\overline{z}} f$, $|\partial_{\overline{z}}f|\leq k|\partial_z f|$ almost everywhere. The measurable function $\mu$ defined by
 \[
\partial_{\overline z} f= \mu \partial_z f
\] 
 is called the \textbf{Beltrami coefficient}. The classical Measurable Riemann Mapping Theorem states that for all $\mu\in L^{\infty}$ with $\|\mu\|_{\infty}<1$, there exists a unique quasiconformal homeomorphism $f^{\mu}\colon \C\to \C$ with Beltrami coefficient $\mu$ and satisfying $f^{\mu}(0)=0$ and $f^{\mu}(1)=1$. The map $f^\mu$ is often called the canonical solution to the Beltrami equation.

For $1<p<\infty$, let $N_p$ denote the operator norm of the Beurling transform $T:L^p(\C)\to L^p(\C)$, defined in Section \ref{sec: elliptic estimates}. The function $p\mapsto N_p$ is continuous in $p$ and satisfies $N_2=1$ (see Proposition \ref{beurlingprop}). We consider the space $L^\infty(\C)\cap W_{\textrm{loc}}^{l,\infty}(\Omega)$ equipped with the semi-norms of $W_{\textrm{loc}}^{l,\infty}(\Omega)$, making it a complex Fr{\'e}chet space, and the open subsets $L_\delta^\infty(\C)\cap W_{\textrm{loc}}^{l,\infty}(\Omega)$ as complex Fr{\'e}chet submanifolds. The main result of the paper is the following.
\begin{theorema}\label{theoremb}
Let $\Omega\subset \C$ be an open set. For $1<p<\infty$, $l\in\mathbb{Z}_{\geq 0},$ and $\eta=\frac{1}{N_p},$ if $\mu \in L_\eta^\infty(\C)\cap W_{\textrm{loc}}^{l,\infty}(\Omega),$ then $f^{\mu}\in W_\textrm{loc}^{l+1,p}(\Omega)$, and the map 
 \begin{align*}
        L_\eta^\infty(\C)\cap W_{\textrm{loc}}^{l,\infty}(\Omega) &\to W_{\textrm{loc}}^{l+1,p}(\Omega)\\
     \mu  &\mapsto f^{\mu}
    \end{align*} is holomorphic.
   \end{theorema}
For $l=0$, this is a classical result of Ahlfors and Bers \cite[Theorem 11]{AB} (a proof is also found in \cite[\S 5.7]{Ast}), which, among other applications, is foundational to Teichm{\"u}ller theory. For $l=\infty,$ it follows that $\mu\mapsto f^\mu$ defines a holomorphic map $L_1^\infty(\C)\cap C^\infty(\Omega) \to C^\infty(\Omega)$ (see Remark \ref{rem: Cinftycase}).

Our motivation for proving Theorem A stems from quasi-Fuchsian geometry. Let $S$ be an oriented surface with complex structure $c_0$, and let $\overline{c_0}$ denote the conjugate complex structure on the oppositely oriented surface $\overline{S}$. To ease the notation, we will use $\widetilde S$ to denote the universal cover of both $S$ and $\overline S$. Assume that the universal cover $(\widetilde S, c_0)$ is biholomorphic to the upper half-space $\mathbb{H},$ and fix for convenience a biholomorphism $\psi_0\colon (\widetilde S,c_0)\to \mathbb H$. The complex conjugate $\overline{\psi_0}$ is a biholomorphism from $(\widetilde{S},\overline{c}_0)$ to $\overline{\mathbb{H}}.$ As well, observe that there exists a representation $\rho_0\colon\pi_1(S)\to \PSL(2,\R)$ for which $\psi_0$ is equivariant: for all $\gamma\in \pi_1(S),$ $\psi_0\circ \gamma=\rho_0(\gamma)\circ \psi_0$.

Let $\mathcal{C}(S,c_0)$ be the space of complex structures $c$ on $S$ such that the identity map $(S,c_0)\to (S,c)$ is quasiconformal (see Section \ref{subs: complex structures}). When $S$ is closed, $\mathcal{C}(S,c_0)$ is just the space of complex structures inducing the same orientation on $S$ and does not depend on $c_0$. For context, the (quasiconformal) Teichm{\"u}ller space $\mathcal{T}(S)$ of $S$ is realized as a quotient of $\mathcal{C}(S,c_0)$. 
Moreover, by taking the Beltrami differential of the quasiconformal mapping, $\mathcal C(S,c_0)$ is in one-to-one correspondence with Beltrami forms on $(S,c_0)$ (see Section \ref{subs: complex structures}). 

For $c_1\in \mathcal{C}(S,c_0)$, we will say that a map $\sigma\colon (\widetilde S,c_1)\to \CP^1$ is \textbf{normalized} if the map $\sigma\circ {\psi_0}^{-1}\colon \mathbb H \to \CP^1$ extends continuously to a boundary map $\partial_\infty \mathbb H \to \CP^1$ that fixes $0,1,$ and $\infty$. Similarly, for $\overline{c_2}\in \mathcal{C}(S,\overline{c_0})$, the map $\sigma\colon (\widetilde S,\overline{c_2})\to \CP^1$ is normalized if $\sigma\circ (\overline{{\psi_0}})^{-1}$ extends analogously. 

A well-known consequence of the Measurable Riemann Mapping Theorem is Bers' Simultaneous Uniformization Theorem (see \cite{SU}), whose proof yields the following statement (see Section \ref{subsection: Bers theorem} for details). 

\begin{theorem} [Bers \cite{SU}]
\label{theorem: Bers theorem}
    For all $(c_1, \overline{c_2})\in \mathcal C(S,c_0)\times \mathcal C(\overline S, \overline{c_0})$, 
there exist two unique disjoint open discs $\Omega_+, \Omega_-\subset \CP^1$ such that $\CP^1\setminus (\Omega_+\sqcup \Omega_-)=\partial \Omega_+=\partial \Omega_-$ is a Jordan curve, and two unique normalized biholomorphisms 
$$\vb*{f_+}(c_1,\overline {c_2})\colon (\widetilde S, c_1) \to \Omega_+,\qquad \qquad  \vb*{\overline{f_-}}(c_1,\overline {c_2})\colon (\widetilde S, \overline{c_2}) \to \Omega_-$$ 
that are equivariant for the same representation $\pi_1(S)\to \PSL(2,\C)$.
\end{theorem}

When $S$ is closed, Bers' Simultaneous Uniformization is often cast as a biholomorphism between $\mathcal{T}(S)\times \mathcal{T}(\overline{S})$ and the space of quasi-Fuchsian representations from $\pi_1(S)$ to $\textrm{PSL}(2,\C)$, which we can read off from the statement above.

In \cite{BEE}, for $S$ closed, the authors use Bers' theorem to associate points in $\mathcal{C}(S)\times \mathcal{C}(\overline{S})$ to complex metrics on $S$ called Bers metrics. We use Theorem A to prove that this association is holomorphic. Things go through more or less the same if we drop the compactness assumption, so we work in generality.

Here is the construction. Given $(c_1,\overline{c_2})\in \mathcal{C}(S,c_0)\times \mathcal{C}(\overline{S},\overline{c_0})$, we package the maps $f_1=\vb*{f_+}\ccpair$ and $\overline{f_2}=\vb*{\overline{f_-}}\ccpair$ from above into a map $$F=(f_1,{\overline{f_2}}): \widetilde S\to \mathbb{CP}^1\times\mathbb{CP}^1\backslash \Delta.$$  The space $\mathbb{CP}^1\times\mathbb{CP}^1\backslash \Delta$ is a holomorphic Riemannian manifold (see, for example, \cite{DZ}), carrying a holomorphic metric 
defined as follows: in any complex affine chart $(U,z)$ on $\CP^1$, the metric on $(U\times U\setminus \Delta, (z_1,z_2))$ has the form 
\begin{equation}
\label{eq: holo metric}
 \inners= -\frac 4 {(z_1-z_2)^2} dz_1\cdot dz_2 \ .
\end{equation}
The Bers metric is the pullback $\vb*{h}_{(c_1,\overline{c_2})}=F^*\inners,$
$$
\hpair= - \frac 4 {(f_1-{\overline{f_2}})^2} d{f_1}\cdot d{\overline{f_2}}.
$$
The equivariance of $f_1$ and $\overline{f_2}$ implies that, as a symmetric complex-valued bilinear form on the complexified tangent bundle $\C T \widetilde S$, $\hpair$ is $\pi_1(S)$-invariant, and hence descends to a bilinear form on $\C T S$. When $c_1=c_2,$ $\vb*{h}_{(c_1,\overline{c_2})}$ is the unique hyperbolic metric in the conformal class prescribed by $c_1=c_2$. That is, $(c_1,\overline{c_2})\mapsto \vb*{h}_{(c_1,\overline{c_2})}$ extends the usual correspondence between complex structures and hyperbolic metrics on $S$. Several geometric interpretations of Bers metrics in relation with Teichm\"uller and quasi-Fuchsian geometry are shown in \cite{ElE}, and applications to the study of equivariant immersions of surfaces in holomorphic Riemannian manifolds, as well as higher Teichm{\"u}ller theory, are studied in \cite{BEE}, \cite{ElSa}, 
\cite{ElSa2}, and \cite{RT}.

As we mentioned before, every element of $\mathcal{C}(S,c_0)$ identifies with a Beltrami form on $(S,c_0)$. For $1\leq l\leq \infty$, we say that a complex structure in $\mathcal{C}(S,c_0)$ is $W_{\textrm{loc}}^{l,\infty}$, if the corresponding Beltrami form is $W_{\textrm{loc}}^{l,\infty}$, and we identify the space of such structures $\mathcal{C}^l(S,c_0)$ with the Fr{\'e}chet space of $W_{\textrm{loc}}^{l,\infty}$ Beltrami forms. When $S$ is closed and $l<\infty$, $\mathcal{C}^l(S,c_0)$ is Banach space. Let $\mathcal{S}(S)$ be the space of complex-valued symmetric bilinear forms on the complexified tangent bundle of $S$, and for $1\leq l\leq \infty$, let $\mathcal{S}^l(S)\subset \mathcal{S}(S)$ be the subspace of $W_{\textrm{loc}}^{l,2}$ bilinear forms. $\mathcal{S}^l(S)$ is a Fr{\'e}chet space, and a Banach space when $S$ is closed and $l<\infty$.
{See Sections \ref{subs: complex structures} and \ref{sec: proof of theorem B} for more information on the Fréchet structures.}

\begin{theoremb}
For all $1\leq l\leq \infty,$ the Bers metric map
    \begin{align*}
        \mathcal{C}(S,c_0)\times \mathcal{C}(\overline{S},\overline{c_0})&\to \mathcal{S}(S)\\
        (c_1,\overline{c_2})&\mapsto \vb*{h}_{(c_1,\overline{c_2})}
    \end{align*}
restricts to a holomorphic map $\mathcal{C}^l(S,c_0)\times \mathcal{C}^l(\overline{S},\overline{c_0})\to \mathcal{S}^{l}(S)$. 
\end{theoremb}
In \cite[Section 6]{BEE}, it is proved that, for $S$ closed and $l=\infty$, the restriction admits a continuous right inverse from an open subset of $\mathcal{S}^{\infty}(S)$ whose points are complex metrics satisfying a non-degeneracy condition called positivity. 
 
In the recent work \cite{ElSa}, one encounters complex elliptic operators involving the Laplacians of Bers metrics. These Laplacians are also considered in \cite{Kim} (compare Theorem B with \cite[Theorem 1.2 (1)]{Kim}). One needs holomorphicity results as above to carry out certain arguments for these operators; for instance, in order to use the holomorphic implicit function theorem. We expect to use Theorems A and B in works to come, and we hope they will be useful to others as well.

\subsection*{Acknowledgements}
We thank Kari Astala for helpful correspondence.  N.S. is funded by the FNR grant O20/14766753, \textit{Convex Surfaces in Hyperbolic Geometry}. At the time of submitting the first version, C.E. was funded by the same FNR grant. C.E. is currently funded by the European Union (ERC, \textit{GENERATE}, 101124349). Views and opinions expressed are however those of the author(s) only and do not necessarily reflect those of the European Union or the European Research Council Executive Agency. Neither the European Union nor the granting authority can be held.

\section{Proof of Theorem A}

To prove Theorem A, we first establish a family of elliptic estimates for the Beltrami operator (Theorem \ref{theorem estimates} below). We then step through the proof from \cite{AB} of the case $l=0$, but use Theorem \ref{theorem estimates} to bootstrap certain estimates along the way.

If a constant $C$ depends on quantities $a_1,\dots, a_n$, we write $C(a_1,\dots, a_n)$. We allow constants to change in the course of a proof. Given a multi-index $\alpha=(\alpha_1,\dots, \alpha_n)$, $\alpha_i\in\mathbb{Z}_{\geq 0}$, we use $|\alpha|$ for the length and $\partial^\alpha$ for the corresponding derivative. As well,  over a domain $\Omega\subset \C,$ we use $||\cdot||_{L^p(\Omega)}$ and $||\cdot||_{W^{k,p}(\Omega)}$ for the $L^p$ and Sobolev norms. When $\Omega=\C,$ we just write $||\cdot||_p$ and $||\cdot||_{k,p}.$

\subsection{Elliptic estimates}\label{sec: elliptic estimates}
The \textbf{Beurling transform} $T$ is defined on $C_0^\infty(\mathbb{C})$ by the principal value integral $$T(u)(z) =\lim_{\varepsilon\to 0} -\frac{1}{\pi}\int_{|\zeta-z|>\varepsilon}\frac{u(\zeta)}{(\zeta-z)^2}dxdy,$$ and the \textbf{Cauchy transform} $P,$ defined on $L^p(\C),$ $2<p<\infty,$ by $$P(u)(z) = -\frac{1}{\pi}\int_{\mathbb{C}}u(\zeta)\Big (\frac{1}{\zeta-z} -\frac{1}{\zeta}\Big) dxdy.$$ 
Via the Calderon-Zygmund theory, the continuity properties of $T$ are well understood. 
\begin{prop}\label{beurlingprop}
    For all $1<p<\infty$, $T$ extends to a continuous linear operator $L^p(\C)\to L^p(\C)$, which is an isometry when $p=2$. The operator norm of $T: L^p(\C)\to L^p(\C)$ varies continuously with $p.$
\end{prop}
For a source, see \cite[Chapter 5]{Ahl}. Recall that we're using $N_p$ to denote the norm of $T:L^p(\C)\to L^p(\C).$ For the Cauchy transform, we take note of the following results, which can be verified immediately. 
\begin{prop}\label{Cauchyprop}
For $2<p<\infty$ and $u\in L^p(\C),$ in the distributional sense, $\partial_{\overline z}(P(u))=u$, and $\partial_z(P(u))=Tu.$
\end{prop}
Our elliptic estimates are as follows. For $r>0,$ let $\mathbb{D}_r=\{z\in\C: |z|<r\}.$
\begin{theorem}
\label{theorem estimates}
Let $\Omega\subset \C$ be an open subset, and let $\mu\in L_1^\infty(\Omega)\cap W_{\textrm{loc}}^{k,\infty}(\Omega)$. For $k\in\mathbb{Z}_{\geq 0}$ and $p$ such that $||\mu||_\infty N_p<1$, let $u\in L_{\mathrm{loc}}^{p}(\Omega)$, and $v\in W_{\mathrm{loc}}^{k,p}(\Omega)$, with $u$ satisfying
    \begin{equation}
        \Big ({\partial}_{\overline{z}}- \mu{\partial}_{ z}\Big )u = v
    \end{equation}
in the weak sense. Then $u\in W_{\mathrm{loc}}^{k+1,p}(\Omega)$, and for all $r<R<\infty$ such that $\mathbb{D}_R\subset \Omega$, there exists $C=C(||\mu||_{W^{k,\infty}(\mathbb{D}_R)},r,R,k,p)$ such that 
$$
    ||u||_{W^{k+1,p}(\mathbb{D}_r)}\leq C(||u||_{W^{k,p}(\mathbb{D}_R)}+||v||_{W^{k,p}(\mathbb{D}_R)}).$$

\end{theorem}
Theorem \ref{theorem estimates} is certainly well known to experts, although we couldn't locate a precise statement. For reference, in Lemmas 5.21 and 5.22 of \cite{Ast}, it is explained how to obtain higher regularity for $u\in W^{1,p}$, assuming $\mu,v\in C_0^\infty(\C)$. As well, Schauder estimates for the Beltrami equation are proved in Chapter 15 of \cite{Ast}.

As is typically the case with interior elliptic estimates, it suffices to first establish the $C_0^\infty$ case, and then one can conclude via standard approximation arguments.

 \begin{lemma}\label{C_0inftycase}
Let $k\in\mathbb{Z}_{\geq 0}$ and let $u,v$ be $C_0^\infty(\C)$ functions 
   with support in a relatively compact open subset $U\subset \C$. Assume that $\mu\in L_1^\infty(U)\cap W^{k,\infty}(U),$ and that $u$ and $v$ satisfy 
   \begin{equation}\label{eqn: belt}
       \Big ({\partial}_{\overline{z}}- \mu{\partial}_{ z}\Big )u = v.
   \end{equation} 
For $p$ such that $||\mu||_\infty N_p<1$, there exists $C=C(||\mu||_{W^{k,\infty}(U)},k,p)$  such that
$$ ||u||_{k+1,p}\leq C(||u||_{k,p}+||v||_{k,p}).$$

\end{lemma}

\begin{proof}
Set $h={\partial}_{\overline{z}}u$, so that by Proposition \ref{Cauchyprop}, $u=Ph$ and ${\partial}_{ z}u=Th$. Using both of these identities, Equation (\ref{eqn: belt}) becomes the integral equation
\begin{equation}\label{integralequation}
    v = {\partial}_{\overline{z}}(Ph) - \mu{\partial}_{ z}(Ph) = h-\mu Th.
\end{equation}
We claim that for all $k\geq 0$, $$||h||_{k,p}\leq C(||\mu||_{W^{k,\infty}(U)},k,p)||v||_{k,p},$$ for $||\mu||_\infty N_p<1,$ and we prove the result by induction on $k.$ For $k=0$, taking $L^p$ norms on (\ref{integralequation}) gives $$||h||_{p} \leq ||v||_{p} + ||\mu||_\infty N_p||h||_{p},$$ and hence for $||\mu||_\infty N_p<1$, we can rearrange to find that $$||h||_p \leq (1-||\mu||_\infty N_p)^{-1}||v||_p.$$ For the induction step, let $\alpha$ be any multi-index with $|\alpha|=k$. Applying $\partial^\alpha$ to (\ref{integralequation}), which commutes with $T,$ we have 
\begin{equation}\label{diffdintegraleq}
    \partial^\alpha v = \partial^\alpha h -\mu T(\partial^\alpha h) + Q(\mu, h),
\end{equation}
where $Q$ is a polynomial in $\partial^\beta \mu$, $\partial^\gamma (Th)=T(\partial^\gamma h),$ for $0\leq |\gamma|<|\alpha|$, $1\leq |\beta|\leq |\alpha|$, $\beta+\gamma=\alpha,$ computed straightforwardly using the product rule. $||Q(\mu,h)||_p$ is rather crudely bounded: $$||Q(\mu,h)||_p \leq \sum_{\beta,\gamma}||\partial^\beta \mu||_{\infty}N_p||\partial^\gamma h||_p\leq C(k)||\mu||_{W^{k,\infty}(U)}N_p||h||_{k-1,p}.$$ Then, using the induction hypothesis $$||h||_{k-1,p}\leq C(||\mu||_{W^{k-1,\infty}(U)},k-1,p)||v||_{k-1,p},$$ together with the trivial bound $||\cdot||_{k-1,p}\leq ||\cdot||_{k,p}$, we extrapolate the final bound $$||Q(\mu,h)||_p\leq  C(||\mu||_{W^{k-1,\infty}(U)},k,p)||\mu||_{W^{k,\infty}(U)}N_p||v||_{k,p}.$$ Returning to (\ref{diffdintegraleq}), we deduce
$$||\partial^\alpha h||_p\leq ||\mu||_{\infty}N_p||\partial^\alpha h||_p +(C(||\mu||_{W^{k-1,\infty}(U)},k,p)||\mu||_{W^{k,\infty}(U)}N_p+1)||v||_{k,p}.$$ Under the assumption $||\mu||_{\infty}N_p<1$, doing the same rearrangement as in the $k=0$ step returns the desired bound for $||\partial^\alpha h||_p$. Summing up over all multi-indices of length $k$ completes the induction.

To finish the proof of the lemma, we turn our estimate on $h$ into an estimate on $u$ using elliptic regularity for ${\partial}_{\overline{z}}$ relative to disks containing the support of $u$. As a first order elliptic operator, the basic elliptic regularity estimate for ${\partial}_{\overline{z}}$ gives that for all $k\in\mathbb{Z}_{\geq 0}$, $1<p<\infty,$ and $u_0\in C_0^\infty(U),$ $$||u_0||_{W^{k+1,p}(U)}\leq C(U,k,p)(||u_0||_{W^{k,p}(U)}+||\partial_{\overline{z}} u_0||_{W^{k,p}(U)}).$$ Taking $u_0=u$, restricting to $p$ such that $||\mu||_{\infty}N_p<1$, and recalling that we defined $h=\partial_{\overline{z}}u$, we obtain 
$$||u||_{k+1,p}\leq C(||u||_{k,p}+||h||_{k,p})\leq C(||u||_{k,p}+||v||_{k,p}),$$ where $C=C(||\mu||_{W^{k,\infty}(U)},k,p)$ (note that we absorbed the dependence of $C$ on $U$ into the dependence on $||\mu||_{W^{k,\infty}(U)}$).
\end{proof}

\begin{proof}[Proof of Theorem \ref{theorem estimates}]
   The result follows swiftly from Lemma \ref{C_0inftycase} and standard approximation arguments, using the density of $C_0^\infty(\C)$ in all of the Sobolev spaces for disks. One can follow the proof of Theorem 10.3.6 and Corollary 10.3.10 in \cite{Ni}, as long as one remembers to restrict the range of $p$ to those such that $||\mu||_\infty N_p<1$. 
\end{proof}

\begin{remark}
   Theorem \ref{theorem estimates} shows that if $f$ is a locally quasiconformal map from a domain $\Omega$ with Beltrami coefficient in $W_{\textrm{loc}}^{k,\infty}(\Omega),$ then $f$ lies in $W_{\textrm{loc}}^{k+1,p}(\Omega)$ for admissible $p.$ The converse is not true: for $q>2,$ the function $f(z) = z +|z|^{2-\frac{2}{q}}$, in a suitably chosen neighbourhood of zero, is a diffeomorphism onto its image. For $p<q/2$, it lies in $W_{\textrm{loc}}^{2,p}$, but the norm of the derivative of the Beltrami coefficient behaves like $|z|^{-\frac{2}{q}}.$
\end{remark}

\subsection{Holomorphic dependence for $L^{\infty}\to W_{\textrm{loc}}^{1,p}$}
\label{holdep}
Moving toward the proof of Theorem A, we first explain the proof of the $l=0$ case from \cite{AB}. In this setting, $\Omega=\C$. For the rest of this section, we fix $1<p<\infty$ and a number $\delta<\frac{1}{N_p}=\eta$, and we assume that all Beltrami coefficients considered here have $L^{\infty}(\C)$ norm at most $\delta;$ this assumption makes the proofs a bit smooth. We explain the holomorphicity for $L_\delta^\infty(\C)\to W_{\textrm{loc}}^{1,p}(\C)$. Since $\delta<\eta$ is arbitrary, the holomorphicity will then hold for $L_\eta^\infty(\C)\to W_{\textrm{loc}}^{1,p}(\C)$ as well. 
\begin{lemma}[Lemma 21 in \cite{AB}]\label{lemma: 21}
Let $\mu\in L_1^\infty(\C)$, $a\in L^\infty(\C)$, and for a real or complex parameter $s$, let $\alpha(s)$ be a continuously varying family of $L^\infty(\C)$ functions such that $||\alpha(s)||_{\infty}$ is uniformly bounded in $s$, and $\lim_{s\to 0}||\alpha(s)||_{\infty}= 0.$ We assume that $s$ is chosen small enough so that $\mu(s)=\mu+sa+s\alpha(s)$ satisfies $||\mu(s)||_{\infty}\leq \delta$. Then,
 $$\lim_{s\to 0}\frac{f^{\mu+sa+s\alpha(s)}-f^\mu}{s}$$ exists in $W_{\textrm{loc}}^{1,p}(\C)$ and does not depend on $\alpha$, and hence defines an element of $W_{\textrm{loc}}^{1,p}(\C).$
\end{lemma}
As in Ahlfors-Bers, we denote the $W_{\textrm{loc}}^{1,p}$-limit by $$\theta^{\mu,a}=\lim_{s\to 0}\frac{f^{\mu+sa+s\alpha(s)}-f^\mu}{s}.$$ After their Lemma 21, the authors show the following. In the lemma below, $\mu,a$, and $\alpha(s)$ are as above.
\begin{lemma}[Part (iii) of Lemma 22 in \cite{AB}]\label{lemma: 22}
Assume we are also given uniformly bounded sequences $(\mu_n)_{n=1}^\infty\subset L_1^\infty(\C)$ and  $(a_n)_{n=1}^\infty\subset L^\infty(\C)$ such that $\lim_{n\to \infty} \mu_n=\mu$ and  $\lim_{n\to\infty} a_n = a$ in $L^\infty(\C)$. Then, $$\lim_{n\to\infty} \theta^{\mu_n,a_n} = \theta^{\mu,a}$$ in $W_{\textrm{loc}}^{1,p}(\C)$.
\end{lemma}
With these two results, they prove that $f^\mu$ depends differentiably on $\mu$, in their chosen Sobolev spaces.
\begin{theorem}[Theorem 10 in \cite{AB}]\label{theorem: 10}
Let $t=(t_1,\dots, t_n)$ and $s=(s_1,\dots, s_n)$ be real or complex parameters, and suppose that we are given a family of Beltrami forms $\mu(t)$ such that for all $t$ in some open set, $$\mu(t+s) = \mu(t) + \sum_{i=1}^n a_i(t)s_i+|s|\alpha(t,s),$$ where $a_1,\dots, a_n$ are uniformly bounded functions in $L^\infty(\C)$ satisfying $a_i(t+s)-a_i(t)\to 0$ in $L^\infty(\C)$ as $s\to 0,$ and  
$\alpha(t,s)$ is a uniformly bounded family of functions in $L^\infty(\C)$ satisfying $\lim_{s\to 0} ||\alpha(t,s)||_{\infty}=0$ for all $t$. Then $f^{\mu(t+s)}$ admits the development 
\begin{equation}\label{eqn: fmu development}
f^{\mu(t+s)}=f^{\mu(t)}+\sum_{i=1}^n\theta^{\mu(t),a_i(t)}s_i + |s|\gamma(t,s),
\end{equation}
 where $\gamma(t,s)$ is a continuously varying family of functions in $W_{\textrm{loc}}^{1,p}(\C)$, uniformly bounded in $W_{\textrm{loc}}^{1,p}(\C)$, and such that, for every $t$, $\gamma(t,s)$ tends to $0$ in $W^{1,p}_{\textrm{loc}}(\C)$ as $s\to 0$.
\end{theorem}
This result is more or less immediate from the two lemmas above. Indeed, Lemma \ref{lemma: 21} shows that one can differentiate $f^{\mu(t)}$ in $t$, which returns $\sum_{i=1}^n\theta^{\mu(t),a_i(t)}$, and Lemma \ref{lemma: 22} shows that the remainder term in the first order Taylor expansion in $W_{\textrm{loc}}^{1,p}(\C)$ has the right decay. See the proof of \cite[Theorem 10]{AB} for the quick proof.

With their Theorem 10, Ahlfors and Bers deduce a holomorphic dependence result (that is, \cite[Theorem 11]{AB}) with ease: assuming that $\mu(t)$ depends holomorphically on complex parameters $(t_1,\dots, t_n),$ then the formula (\ref{eqn: fmu development}) shows that, as an element of $W_{\textrm{loc}}^{1,p}(\C),$ $f^{\mu(t)}$ satisfies the Cauchy-Riemann equations in every $t_1,\dots, t_n$, and hence $t\mapsto f^{\mu(t)}$ is holomorphic in $W_{\textrm{loc}}^{1,p}(\C).$

Finally, we have to address in what sense the Ahlfors-Bers map is holomorphic. There are different notions of holomorphicity on locally convex topological vector spaces, and \cite[Theorem 11]{AB} shows that if $\Lambda$ is any finite dimensional complex manifold, and $\Lambda\mapsto L_\delta^\infty(\C)$, $t\mapsto\mu(t) $, is holomorphic, then $t\mapsto f^{\mu(t)}$ is holomorphic into $W_{\textrm{loc}}^{1,p}(\C).$ In particular, the global map $L_\delta^\infty(\C)\to W_{\textrm{loc}}^{1,p}(\Omega)$, as a map between complex Fr{\'e}chet spaces, is Gateaux holomorphic, which means that it is holomorphic in restriction to every complex line. While Gateaux holomorphicity does not always imply holomorphicity in the usual sense, i.e., Fr{\'e}chet holomorphicity (see \cite[Chapters 1 and 2]{Dodson}), it is a theorem attributed to Graves-Taylor-Hille-Zorn that a locally bounded and Gateaux holomorphic map between lots of classes of complex spaces is (Fr{\'e}chet) holomorphic (note that, even in finite dimensions, the analogous result for real manifolds is false). See \cite[Chapter 14]{chae1985holomorphy} for theory and exposition around these notions of holomorphicity, and the Graves-Taylor-Hille-Zorn theorem. Note that, in our source \cite{chae1985holomorphy}, the most general result is stated for open subsets of complex Banach spaces, but by taking projective limits it extends to Fr{\'e}chet spaces \cite{Dodson}, and since holomorphicity is local, the result extends further to Fr{\'e}chet manifolds. By \cite[Lemma 6]{AB}, the map in question is continuous, thus locally bounded, and hence the theorem is proved for $l=0.$

\subsection{Main proof}
We continue to work under the assumption that all Beltrami coefficients in $ L^\infty(\C)$ have norm at most $\delta<\eta$. Similar to above, we show holomorphicity for $L_\delta^\infty(\C)\cap W_{\textrm{loc}}^{l,\infty}(\Omega)\to  W_{\textrm{loc}}^{l+1,p}(\Omega),$ from which the holomorphicity for $L_\eta^\infty(\C)\cap W_{\textrm{loc}}^{l,\infty}(\Omega)\to  W_{\textrm{loc}}^{l+1,p}(\Omega)$ follows. 

To prove Theorem A, we need to establish the analogue of Theorem 10 from \cite{AB}. We package the needed ingredients into the lemma below. Given $\mu\in L_1^\infty(\C)$ and $a\in L^\infty(\C)$, we keep the same definition of $\theta^{\mu,a},$ which we already know defines an element of $W_{\textrm{loc}}^{1,p}(\C)$. Let $\Omega\subset \C$ be the open subset of Theorem A.
\begin{lemma}\label{lemma: higher Sobolev improvement}
    In the setting of Lemma 21 from \cite{AB}, now assume that $\mu,a,\alpha(s) \in W_{\textrm{loc}}^{l,\infty}(\Omega),$ that $s\mapsto \alpha(s)$ is continuous as a map to $W_{\textrm{loc}}^{l,\infty}(\Omega)$, and that on any relatively compact $U\subset \Omega$, $||\alpha(s)||_{W^{l,\infty}(U)}$ is uniformly bounded in $s$, and tends to $0$ as $s\to 0.$ Then, 
    \begin{enumerate}
        \item $\theta^{\mu,a}\in W_{\textrm{loc}}^{l+1,p}(\Omega)$,
        \item  and $$\theta^{\mu,a}=\lim_{s\to 0}\frac{f^{\mu+sa+s\alpha(s)}-f^\mu}{s}$$ in $W_{\textrm{loc}}^{l+1,p}(\Omega).$
    \end{enumerate}
    Moreover, given uniformly bounded sequences $(\mu_n)_{n=1}^\infty\subset L_1^\infty(\C)\cap W_{\textrm{loc}}^{l,\infty}(\Omega)$ and  $(a_n)_{n=1}^\infty\subset W_{\textrm{loc}}^{l,\infty}(\Omega)$ such that $\lim_{n\to \infty} \mu_n=\mu$ and $\lim_{n\to \infty}a_n =a$ in $W_{\textrm{loc}}^{l,\infty}(\Omega),$ we have that $$\lim_{n\to\infty} \theta^{\mu_n,a_n} = \theta^{\mu,a}$$ in $W_{\textrm{loc}}^{l+1,p}(\Omega)$.
\end{lemma}
Theorem A follows from Lemma \ref{lemma: higher Sobolev improvement} by exactly the same argument as in \cite{AB}: if $\mu(t)$ depends on a complex parameter $t=(t_1,\dots, t_n),$ then for small $s=(s_1,\dots, s_n)$ we can now write 
\begin{equation}\label{eqn: better development}
    f^{\mu(t+s)}=f^{\mu(t)}+\sum_{i=1}^n\theta^{\mu(t),a_i(t)}s_i + |s|\gamma(t,s),
\end{equation}
with every term the same function as in Theorem \ref{theorem: 10}, but also lying in $W_{\textrm{loc}}^{l+1,p}(\Omega),$ and with $\gamma(t,s)$ tending to $0$ in $W_{\textrm{loc}}^{l+1,p}(\Omega)$ as $s\to 0$ for every $t$. That is, $f^{\mu(t)}$ depends differentiably in $t$ as an element of $W_{\textrm{loc}}^{l+1,p}(\Omega)$. And, as in \cite[Theorem 11]{AB}, if the  $\mu(t)$'s vary holomorphically in $W_{\textrm{loc}}^{l,\infty}(\Omega),$ then by the formula (\ref{eqn: better development}), $f^{\mu(t)}$ satisfies the Cauchy-Riemann equations in $W_{\textrm{loc}}^{l+1,p}(\Omega),$ so as in the $l=0$ case we have Gateaux holomorphicity. Combining \cite[Lemma 6]{AB} and Theorem \ref{theorem estimates} returns that our map from $L_\delta^\infty(\C)\cap W_{\textrm{loc}}^{l,\infty}(\Omega)\to W_{\textrm{loc}}^{l+1,p}(\Omega)$ is locally bounded (and in fact, it is not difficult to use these same results to see that it is continuous), and hence, analogous to the $l=0$ case, the Graves-Taylor-Hille-Zorn theorem yields the statement of Theorem A. So, it remains to prove Lemma \ref{lemma: higher Sobolev improvement}.
\begin{proof}[Proof of Lemma \ref{lemma: higher Sobolev improvement}]
We bootstrap on the estimates implied by Lemmas \ref{lemma: 21} and \ref{lemma: 22}, using Theorem \ref{theorem estimates}. To show our bounds and convergence in $W_{\textrm{loc}}^{1,p}(\Omega),$ it suffices to prove that for any fixed disk $U\subset \Omega$, we have bounds and convergence in $W^{l+1,p}(U)$. Recall that we declared $p$ to be close enough to $2$ so that $\delta N_p<1$, and hence in particular we can always apply Theorem \ref{theorem estimates}. As well, choosing $s$ small, we can always assume that $f^{\mu+sa+s\alpha(s)}$ is uniformly bounded in $W_{\textrm{loc}}^{l+1,p}(\Omega)$, independently of $s.$

Note that (as is observed in \cite{AB}), 
\begin{equation}\label{eqn: derivative of first variation}
    (\partial_{\overline{z}}-\mu\partial_z) \theta^{\mu,a}= a \partial_z f^\mu.
\end{equation}
Thus, since $a\in W_{\textrm{loc}}^{l,\infty}(\Omega)$ and $f^{\mu}\in W_{\textrm{loc}}^{l+1,p}(\Omega)$, $$(\partial_{\overline{z}}-\mu\partial_z) \theta^{\mu,a}\in W_{\textrm{loc}}^{l,p}(\Omega).$$
The first assertion of Lemma \ref{lemma: higher Sobolev improvement} then follows routinely: by Theorem \ref{theorem estimates}, for any disk $V\subset \Omega$ containing $U$, $$||\theta^{\mu,a}||_{W^{2,p}(U)}\leq C(||\mu||_{W^{1,\infty}(V)},U,V,1,p)(||\theta^{\mu,a}||_{W^{1,p}(V)}+ ||(\partial_{\overline{z}}-\mu\partial_z) \theta^{\mu,a}||_{W^{1,p}(V)}),$$ so $\theta^{\mu,a}\in W_{\textrm{loc}}^{2,p}(\Omega)$. Bootstrapping that estimate as much as we can, we obtain $\theta^{\mu,a}\in W_{\textrm{loc}}^{l+1,p}(\Omega)$. Note that, as we do the bootstrap, the dependency in the final estimate of $||\theta^{\mu,a}||_{W^{l+1,p}(U)}$ is on $||\mu||_{W^{l,\infty}(V)}, U, V, l$, and $p$.

For the second assertion, setting $$\hat{\chi}(s) = \frac{f^{\mu+sa+s\alpha(s)}-f^\mu}{s},$$ observe that  
$$(\partial_{\overline{z}}-\mu\partial_z)\hat{\chi}(s)=(a+\alpha(s))\partial_z f^{\mu+sa+s\alpha(s)},$$ which is uniformly bounded in $W_{\textrm{loc}}^{l,p}(\Omega)$, since for small $s$,
$$||(a+\alpha(s))\partial_z f^{\mu+sa+s\alpha(s)}||_{W^{l,p}(U)}\leq (||a||_{W^{l,\infty}(U)}+1)||f||_{W^{l+1,p}(U)}.$$ Therefore, by Theorem \ref{theorem estimates}, since we know by Lemma \ref{lemma: 21} that $\hat{\chi}(s)$ is uniformly bounded in $W_{\textrm{loc}}^{1,p}(\Omega)$, we get that it is uniformly bounded in $W_{\textrm{loc}}^{2,p}(\Omega)$. Iterating the bootstrap as above, we obtain that $\hat{\chi}(s)$ is uniformly bounded in $W_{\textrm{loc}}^{l+1,p}(\Omega)$. Moreover, the function
$$\chi(s) = \frac{f^{\mu+sa+s\alpha(s)}-f^\mu}{s}-\theta^{\mu,a},$$ which we know tends to $0$ in $W_{\textrm{loc}}^{1,p}(\Omega)$ as $s\to 0,$ is uniformly bounded in $W_{\textrm{loc}}^{l+1,p}(\Omega)$. We further compute that
$$(\partial_{\overline{z}}-\mu\partial_z)\chi(s) =
(a+\alpha(s))\partial_z f^{\mu+sa+s\alpha(s)} - a\partial_z f^{\mu},$$ from which we deduce that $(\partial_{\overline{z}}-\mu\partial_z)\chi(s)$ tends to $0$ in $W_{\textrm{loc}}^{l,p}(\Omega)$ as $s\to 0,$ since 
\begin{align*}
    ||(\partial_{\overline{z}}-\mu\partial_z)\chi||_{W^{l,p}(U)} &\leq ||\alpha(s)||_{W^{l,\infty}(U)}||f^{\mu+sa+s\alpha(s)}||_{W^{l+1,p}(U)}+||a||_{W^{l,\infty}(U)}(||f^{\mu+sa+s\alpha(s)}-f^{\mu}||_{W^{l+1,p}(U)}) \\
    &= ||\alpha(s)||_{W^{l,\infty}(U)}||f^{\mu+sa+s\alpha(s)}||_{W^{l+1,p}(U)}+||a||_{W^{l,\infty}(U)}|s|\cdot ||\hat{\chi}||_{W^{l+1,p}(U)},
\end{align*}
and both terms above tend to $0$ as $s\to 0.$ Bootstrapping via Theorem \ref{theorem estimates} again, we deduce that $\chi(s)$ tends to $0$ in $W_{\textrm{loc}}^{1+1,p}(\Omega)$ as $s\to 0,$ which proves the second assertion.

Finally, for the ``moreover" statement, we compute, using (\ref{eqn: derivative of first variation}), that
\begin{align*}
    (\partial_{\overline{z}}-\mu\partial_z)(\theta^{\mu,a}-\theta^{\mu_n,a_n}) &=  (\partial_{\overline{z}}-\mu\partial_z)\theta^{\mu,a}- (\partial_{\overline{z}}-\mu_n\partial_z)\theta^{\mu_n,a_n}-(\mu_n-\mu)\partial_z \theta^{\mu_n,a_n} \\
    &=a\partial_z f^\mu - a_n\partial_z f^{\mu_n} - (\mu_n-\mu)\partial_z \theta^{\mu_n,a_n},
\end{align*}
and hence for $n$ large, $$ ||(\partial_{\overline{z}}-\mu\partial_z)(\theta^{\mu,a}-\theta^{\mu_n,a_n})||_{W^{l,p}(U)}\leq (||a||_{W^{l,\infty}(U)}+1)||f^\mu - f^{\mu_n}||_{W^{l,p}(U)}+ ||\mu_n - \mu||_{W^{l,\infty}(U)}||\theta^{\mu_n,a_n}||_{W^{l+1,p}(U)}.$$ Since the bounds on $||\theta^{\mu_n,a_n}||_{W^{l+1,p}(U)}$ depend only on the $W_{\textrm{loc}}^{l,\infty}(\Omega)$ semi-norms of $\mu_n$ and $a_n,$ they are uniformly controlled. It follows that $$\lim_{n\to \infty} ||(\partial_{\overline{z}}-\mu\partial_z)(\theta^{\mu,a}-\theta^{\mu_n,a_n})||_{W^{l,p}(U)}=0.$$ By yet another bootstrap using Theorem \ref{theorem estimates}, we conclude that $\theta^{\mu_n,a_n}$ tends to $\theta^{\mu,a}$ in $W_{\textrm{loc}}^{l+1,p}(\Omega)$, as desired. 
\end{proof}
As discussed above, verifying Lemma \ref{lemma: higher Sobolev improvement} completes the proof of Theorem A.
\begin{remark}\label{rem: Cinftycase}
    For the $C^\infty$ case, we take $p=2$, and then for every $l$, the map $\mu\mapsto f^\mu$ defines a holomorphic map $L_1^\infty(\C) \cap W_{\textrm{loc}}^{l,\infty}(\Omega)\to W_{\textrm{loc}}^{l+1,2}(\Omega)$. Since the Fr{\'e}chet structure on $C^\infty(\Omega)$ can be generated by the semi-norms coming from $W_{\textrm{loc}}^{l,\infty}(\Omega)$ on the source, as well as $W_{\textrm{loc}}^{l+1,2}(\Omega)$ on the target, it follows that $\mu\mapsto f^\mu$ restricts to a holomorphic map from $L_1^\infty(\mathbb{C})\cap C^\infty(\Omega)\to C^\infty(\Omega).$
\end{remark}

\section{Proof of Theorem B}
\label{sec: Bers theorem}

\subsection{Quasiconformal complex structures on a surface}
\label{subs: complex structures}
Let $S$ be an oriented surface and let $c_0$ be a complex structure inducing the same orientation and such that the holomorphic universal cover is biholomorphic to the upper half-plane $\mathbb H\subset \C$. As in the introduction, we fix such a biholomorphism $\psi_0: (\widetilde S, c_0)\to \mathbb H$, which by construction is equivariant.

A measurable section $\mu$ of the bundle $T^{1,0}S\otimes (T^{0,1}S)^*$ can be written locally in the form $\mu(z) \frac {d\overline z}{dz}.$ While the function $\mu(z)$ depends on the choice of local coordinate, its absolute value does not, and hence $|\mu(z)|$ determines a well-defined measurable function on $S$, which we denote by $|\mu|$. A \textbf{Beltrami form} on $(S,c_0)$ is a measurable section $\mu$ of $T^{1,0}S\otimes (T^{0,1}S)^*$ such that $|\mu|$ lies in the open unit ball $L^\infty_1(S)$ of $L^\infty(S)$. Pulling back forms through the universal cover and $\psi_0$ (together with the choice of the canonical coordinate $z$ on $\mathbb H$) allow one to see the space of Beltrami forms on $(S,c_0)$ as a {closed} complex subspace $\textrm{Belt}_{c_0}(S)$ of $L_1^\infty(\mathbb H)$ (namely, those satisfying the correct transformation law with respect to the $\pi_1(S)$-action). If we instead work on $(\overline{S},\overline{c_0}),$ the space of Beltrami forms $\textrm{Belt}_{\overline{c_0}}(\overline{S})$ is seen as a complex subspace of $L_1^\infty(\overline{\mathbb H})$.

Let $f\colon(S,c)\to (S,c')$ be an orientation preserving homeomorphism that lies in $W_{\textrm{loc}}^{1,2}$. Working in local coordinates for $c$ and $c'$, the expressions $\frac{\partial_{\overline z} f d\overline z}{\partial_z f dz}$, where derivatives are taken in the distributional sense, determine a measurable section $\mu$ of $T^{1,0}S\otimes (T^{0,1}S)^*$. If $\mu$ is a Beltrami form, then the map $f$ is called \textbf{quasiconformal}, and $\mu$ is moreover called the \textbf{Beltrami differential of $f$}.

Now, as in the introduction, let $\mathcal{C}(S,c_0)$ be the space of complex structures $c$ on $S$ such that the identity map $(S,c_0)\to (S,c)$ is quasiconformal. The space $\mathcal{C}(S,c_0)$ is in one-to-one correspondence with the space of Beltrami forms on $(S,c_0)$. Indeed, to every $c\in \mathcal{C}(S,c_0)$, one can associate the Beltrami differential of the identity map $id\colon (S,c_0)\to (S,c)$. To construct the inverse, let $\mu\in \mathrm{Belt}_{c_0}(S)$ and extend it to the lower half-plane $\overline {\mathbb H}$ by reflection $\overline z \mapsto \overline {\mu (z)}$. The Measurable Riemann Mapping Theorem provides a quasiconformal homeomorphism $f_{\mu}\colon \mathbb H\to \mathbb H$ with Beltrami coefficient $\mu$. Then, by construction, the pullback complex structure through $f_\mu\circ \psi_0$ is preserved by the deck group, so it descends to a complex structure on $S$. Observe that this argument also shows that the holomorphic universal cover of $(S,c)$ is conformal to $\mathbb H$. 

If $S$ is closed and of genus at least 2, $\mathcal C(S,c_0)$ does not depend on the choice of $c_0$: it is the whole space of complex structures on $S$ inducing the same orientation as $c_0$. 

For $1\le l \leq \infty$, as in the introduction, we define $\mathcal C^l(S,c_0)$ as the space of complex structures such that, for the base point $(S,c_0)$, correspond to $W_{\textrm{loc}}^{l, \infty}$ elements in $\mathrm{Belt}_{c_0}(S)$, with the understanding that for $l=\infty$, $W_{\textrm{loc}}^{l, \infty}=C^\infty$ (note that, when $S$ is closed, we can drop the ``loc"). For every $l$, $\mathcal{C}^{l}(S,c_0)$ is a complex Fr{\'e}chet manifold, while for $S$ closed and $l<\infty$, $\mathcal{C}^l(S,c_0)$ is a complex Banach manifold. 

Since we are concerned with holomorphic maps involving these spaces, we take a moment to describe the Fréchet manifold structures. We assume that $S$ is not compact for now. Similar to the case $l=0$, for every $l$, we view $\mathcal C^l(S,c_0)$ as the intersection of $L_1^\infty(\mathbb{H})$ with a closed complex subspace of $W_{\textrm{loc}}^{l,\infty}(\mathbb{H})$. Once we have a Fréchet space structure on $W_{\textrm{loc}}^{l,\infty}(\mathbb{H})\cap L^\infty(\mathbb{H})$ such that the $L^\infty$ norm is continuous, $\mathcal C^l(S,c_0)$ becomes an open subset of a Fréchet space, hence a Fréchet manifold. For $l<\infty,$ the semi-norms that generate the Fréchet structure on $W_{\textrm{loc}}^{l,\infty}(\mathbb{H})$ are given by taking the $W^{l,\infty}$ norms over a covering of $\mathbb{H}$ by a countable family of relatively compact open sets. For $l=\infty$, one usually takes the semi-norms to be the $C^k$ norms, $k\in \mathbb{Z}_{\geq 0}$, over a relatively compact covering. We point out that, by the Sobolev embedding theorem, if we choose our open subsets to have Lipschitz boundary (this condition ensures that the Sobolev embedding theorem can be applied) one can replace the $C^k$ norms, $k\in \mathbb{Z}_{\geq 0}$, with the $W^{l,\infty}$ norms, $l\in \mathbb{Z}_{\geq 0}$, and the Fréchet space structure will be isomorphic. When $S$ is closed, for $l<\infty$, one can define a Banach space structure by taking the $W^{l,\infty}$ norm over a fundamental domain, and, when $l=\infty,$ one takes the $W^{l,\infty}$ norms as the generating semi-norms. 
\subsection{Bers' theorem}\label{subsection: Bers theorem}

We briefly explain how to prove Theorem \ref{theorem: Bers theorem} from the Measurable Riemann Mapping Theorem. The procedure detailed below is not quite the same as in Bers' seminal paper \cite{SU}, but it is equivalent and has been carried out as below in other sources, for instance in \cite[\S 6.12]{Hub}. Given $(c_1,\overline{c_2})\in\mathcal{C}(S, c_0)\times \mathcal{C}(\overline{S}, \overline{c_0})$, let $\mu_1$ and $\mu_2$ be the elements of $\textrm{Belt}_{c_0}(S)\subset L_1^\infty(\mathbb H)$ and $\textrm{Belt}_{\overline{c_0}}(\overline{S})\subset L_1^\infty(\overline{\mathbb H})$ representing $c_1$ and $c_2$ respectively. Then, applying the Measurable Riemann Mapping Theorem to
\[
\widehat \mu(z)= \begin{cases}
    \mu_1(z) &\text{ if $z\in \mathbb H$}\\
    \mu_2(z) &\text{ if $z\in \overline{\mathbb H}$ }, 
\end{cases}
\]
we obtain the quasiconformal map $f^{\widehat{\mu}},$ which we view as a homeomorphism of $\mathbb{CP}^1$. Note that, although $\widehat{\mu}$ isn't defined on the measure zero set $\mathbb{R}$, the canonical solution depends only on the element of $L^\infty(\C)$ determined by $\widehat{\mu}$. Under the identifications $\psi_0: (\widetilde S,c_0)\to \mathbb H$ and $\overline{\psi_0}:(\widetilde S,\overline{c_0})\to \overline{\mathbb H}$, the maps $f_1=\vb*{{f_+}}\ccpair$ and $\overline{f_2}=\vb*{\overline{f_-}}\ccpair$ from the statement of Theorem \ref{theorem: Bers theorem} are defined, for $z\in \mathbb{H}$, by $f_1(z) = f^{\widehat{\mu}}(z)$ and $\overline{f_2}(\overline{z}) = f^{\widehat{\mu}}(\overline{z})$. The equivariance follows from the uniqueness of the solution to the Beltrami equation.

\subsection{Composition maps}
We are just about ready to prove Theorem B, but first we need a general result on composition maps between Sobolev spaces. Composition maps between various function spaces have been thoroughly investigated; see, for instance, \cite[Chapter 5]{Sobolevbook}. We could not find statements in the literature quite close enough to the results below.

Let $k\geq 1$ be an integer and let $1<p<\infty$. Given open subsets $\Omega\subset \C^m$ and $U\subset \C^n$, we view $W_{\textrm{loc}}^{k,p}(\Omega,U)$, the space of $W_{\textrm{loc}}^{k,p}(\Omega)$ functions that map into $U$, as an open subset of $W_{\textrm{loc}}^{k,p}(\Omega)$, and hence it is a Fréchet manifold. When $\frac{m}{p}<1,$ the Sobolev embedding theorem gives, for any $s\geq 1$, a continuous inclusion $W_{\textrm{loc}}^{s,p}\to W_{\textrm{loc}}^{s-1,\infty},$ and when $k-\frac{m}{p}>0$ (which holds when $\frac{m}{p}<1$), the space $W_{\textrm{loc}}^{k,p}(\Omega)$ is a Fréchet algebra. In this subsection, we resume all of our notation related to Sobolev spaces and multi-indices from Section 2. 

\begin{prop}
\label{prop: loc composition is holo}
    Let $\Omega\subset \C^m$ and $U\subset \C^n$ be open subsets and let $g: U\to \C$ be a $C^\infty$ function. Assuming $\frac{m}{p}<1$, the map $$\Phi_g: W_{\textrm{loc}}^{k,p}(\Omega,U)\to W_{\textrm{loc}}^{k,p}(\Omega), \hspace{1mm} u\mapsto g\circ u,$$ is well-defined, i.e., $g\circ u\in W_{\textrm{loc}}^{k,p}(\Omega)$, and Fréchet differentiable. Furthermore, the map is Fréchet holomorphic if and only if $g$ is holomorphic in $U$. 
\end{prop}
In the proof of Theorem B, we will use Proposition \ref{prop: loc composition is holo} to show that pulling back the holomorphic Riemannian metric on $\mathbb{CP}^1\times\mathbb{CP}^1\setminus \Delta$ is a holomorphic operation between appropriate Fréchet manifolds.

Since the Fréchet structure on $W_{\textrm{loc}}^{k,p}(\Omega)$ is generated by $W^{k,p}$ norms with respect to open subsets taken over a relatively compact covering of $\Omega$, we first treat the relatively compact case (Lemma \ref{lemma: composition holomorphic} below). Before getting to that, we need a general boundedness result.
\begin{lemma}\label{lemma: composition bounded}
    Let $\Omega\subset \C^m$ and $U\subset \C^n$ be relatively compact open subsets with Lipschitz boundary, and let $g_0: U\to \C$ be a $C^\infty$ function that extends smoothly to a neighbourhood of the closure of ${U}$.  Assuming $\frac{m}{p}<1$, the corresponding map $$\Phi_{g_0}: W^{k,p}(\Omega,U)\to W^{k,p}(\Omega), \hspace{1mm} u\mapsto g_0\circ u,$$ is well-defined and satisfies a bound 
    \begin{equation}\label{eq: bound for C_g}
        ||\Phi_{g_0}(u)||_{W^{k,p}(\Omega)}\leq C(k,p,\Omega,U)(||g_0||_{L^\infty(U)}+P(||g_0||_{W^{k,\infty}(U)},||u||_{W^{k,p}(\Omega)})),
    \end{equation}
for $P$ some polynomial of $||g_0||_{W^{k,\infty}(U)}$ and $||u||_{W^{k,p}(\Omega)}$ that vanishes when $u=0.$ 
\end{lemma}
%I've jigged around assumptions on k, etc., reminder to check them.

The reason we replace $g$ with $g_0$ is to not cause confusion: during the proof of Lemma \ref{lemma: composition holomorphic} (a special case of Proposition \ref{prop: loc composition is holo}), we will apply Lemma \ref{lemma: composition bounded} to auxiliary functions.

\begin{proof}
Write out 
$$||\Phi_{g_0}(u)||_{W^{k,p}(\Omega)}=\sum_{|\alpha|\leq k} ||\partial^\alpha (\Phi_{g_0}(u))||_{L^p(\Omega)}.$$ The $\alpha=0$ term is trivially bounded by $C(p,U)||g_0||_{L^\infty(U)}$. Controlling the other terms in $||\Phi_{g_0}(u)||_{W^{k,p}(\Omega)}$ is a straightforward application of the Fa{\`a} di Bruno formula, which says here that for every multi-index $\alpha$ of length at most $k$ and smooth $u\in W^{k,p}(\Omega,U)$,
$$\partial^\alpha (g_{0}\circ u)(z)=P^\alpha(((\partial^\gamma g_{0})(u(z)))_{|\gamma|\leq |\alpha|}, (\partial^\gamma u(z))_{|\gamma|\leq |\alpha|}),$$
 where $P^\alpha$ is a linear combination with non-negative coefficients in the $(\partial^\gamma g_{0})(u(z))$'s and $\partial^\gamma u(z)$'s. We prove (\ref{eq: bound for C_g}) for a smooth $u\in W^{k,p}(\Omega,U)$, and then the general result follows using the density of smooth functions in the Sobolev space.
 
 Let $z_1,\dots, z_m$ be the complex coordinates on $\Omega$ and $\zeta_1,\dots, \zeta_n$ the complex coordinates on $U$. 
It is checked directly that
$$\partial^\alpha (g_{0}\circ u)(z)= Q^\alpha(((\partial^\gamma g_{0})(u(z)))_{|\gamma|\leq |\alpha|}, (\partial^\gamma u(z))_{|\gamma|<|\alpha|}))+ K^\alpha(g_0,u)(z),$$ for $Q^\alpha$ some linear combination with non-negative coefficients as above and  
$$K^\alpha(g_0,u)(z):=\sum_{i=1}^n\Big (\frac{\partial g_0}{\partial \zeta_i}(u(z)) \partial^\alpha u_i(z) + \frac{\partial g_0}{\partial \overline{\zeta}_i}(u(z)) \partial^\alpha \overline{u}_i(z)\Big ).$$
Taking norms, we have the bound
$$|\partial^\alpha (g_{0}\circ u)|\leq  Q^\alpha((||g_0||_{W^{s,\infty}(U)})_{s\leq k}, (||u(z)||_{W^{s,\infty}(\Omega)})_{s<k})+ ||g_0||_{W^{1,\infty}(U)}(|\partial^\alpha u| + |\partial^\alpha \overline{u}|).$$
Using our assumption that $\frac{m}{p}<1$, for all $s,$ $W^{s+1,p}$ embeds continuously into $W^{s,\infty}$, and hence 
$$|\partial^\alpha (g_{0}\circ u)|\leq  Q^\alpha((||g_0||_{W^{s,\infty}(U)})_{s\leq k}, (||u(z)||_{W^{s+1,p}(\Omega)})_{s<k})+ ||g_0||_{W^{1,\infty}(U)}(|\partial^\alpha u| + |\partial^\alpha \overline{u}|).$$
It is clear that the $L^p$ norm of $(|\partial^\alpha u| + |\partial^\alpha \overline{u}|)$ is controlled by the $W^{k,p}(\Omega)$ norm of $u$. Putting it all together, the $L^p(\Omega)$ norm of $\partial^\alpha (g_{0}\circ u)$ is bounded by a polynomial in terms of $||g_0||_{W^{k,\infty}(U)}$ and $||u||_{W^{k,p}(\Omega)}.$ To get the bound (\ref{eq: bound for C_g}), sum up the bounds above for every non-zero $\alpha.$ As discussed above, handling the smooth case suffices for the general result.
\end{proof}

With Lemma \ref{lemma: composition bounded} in hand, we come to the relatively compact version of Proposition \ref{prop: loc composition is holo}. In the proof below, for a multi-index $\beta=(\beta_1,\dots, \beta_n)$ and $v=(v_1,\dots, v_n)\in \C^n$, we set $v^\beta =v_1^{\beta_1}\dots v_n^{\beta_n}.$

\begin{lemma}\label{lemma: composition holomorphic}
   Let $\Omega\subset \C^m$ and $U\subset \C^n$ be relatively compact open subsets with Lipschitz boundary, and let $g: U\to \C$ be a $C^\infty$ function that extends smoothly to a neighbourhood of the closure of ${U}$. Assuming $\frac{m}{p}<1$, the map $$\Phi_g: W^{k,p}(\Omega,U)\to W^{k,p}(\Omega), \hspace{1mm} u\mapsto g\circ u,$$ is well-defined and Fréchet differentiable. Furthermore, the map is Fréchet holomorphic if and only if $g$ is holomorphic in $U$. 
\end{lemma}

\begin{proof}
    As a Fréchet manifold, the tangent space of $W^{k,p}(\Omega,U)$ at any point identifies with $W_0^{k,p}(\Omega)$, i.e., the $W^{k,p}(\Omega)$ functions that map to $0$ under the trace operator. Let $u\in W^{k,p}(\Omega,U)$ and $h\in W_0^{k,p}(\Omega)$. By the Sobolev embedding theorem, $u$ is continuous. By Taylor's Theorem with Remainder, as a function on $\Omega,$ $$g(u+h)(z) = g(u)(z) + dg_{u(z)}(h(z))+\sum_{|\beta|=2}R_\beta(u,h)(z)(h(z))^\beta,$$ where each $\beta$ is a multi-index corresponding to complex coordinates on $U$, and
    $$R_\beta(u,h)=\int_0^1(1-t)(\partial^\beta g)\circ (u(z)+th(z)) dt.$$
We claim that 
\begin{enumerate}
    \item the $\R$-linear map $dg_u: W_0^{k,p}(\Omega)\to W^{k,p}(\Omega),$ $h\mapsto dg_u(h),$ where $(dg_u(h))(z)=dg_{u(z)}(h(z)),$ is well-defined and continuous, and  
    \item for every $\beta$ with $|\beta|=2,$ $R_\beta(u,h)(z)$ is in $W^{k,p}(\Omega)$ and has norm controlled linearly by $||u||_{W^{k,p}(\Omega)}$ and $||h||_{W^{k,p}(\Omega)}.$ 
\end{enumerate}
Granting these claims, it follows easily that $\Phi_g$ is Fréchet differentiable with Fréchet derivative at a point $u$ given by the $\R$-linear map $dg_u$. Indeed, since, under our assumptions, $W^{k,p}(\Omega)$ is a Banach algebra, 
$$\frac{||g(u+h)-g(u)-dg_u(h)||_{W^{k,p}(\Omega)}}{||h||_{W^{k,p}(\Omega)}} \leq \sum_{|\beta|=2}\frac{||R_\beta(u,h)||_{W^{k,p}(\Omega)}||h^\beta||_{W^{k,p}(\Omega)}}{||h||_{W^{k,p}(\Omega)}},$$ and it is easily checked that $$\lim_{||h||_{W^{k,p}(\Omega)}\to 0} \frac{||h^\beta||_{W^{k,p}(\Omega)}}{||h||_{W^{k,p}(\Omega)}} =0.$$

Beginning with the justification of claim (1), since $dg_u$ is linear, we only need to show that it is continuous at $0$. In coordinates $z_1,\dots z_m$ on $\Omega$ and $\zeta_1,\dots, \zeta_n$ on $U$, $$dg_{u(z)}(h(z)) = \sum_{i=1}^n \Big ( \frac{\partial g}{\partial \zeta_i}(u(z)) h_i(z)+ \frac{\partial g}{\partial \overline{\zeta}_i}(u(z)) \overline{h}_i(z)\Big ).$$
Using the Banach algebra property, $$||dg_{u(z)}(h(z))||_{W^{k,p}(\Omega)}\leq C\sum_{i=1}^n \Big( \Big|\Big | \frac{\partial g}{\partial \zeta_i}(u(z))\Big |\Big |_{W^{k,p}(\Omega)}+  \Big|\Big | \frac{\partial g}{\partial \overline{\zeta}_i}(u(z))\Big|\Big|_{W^{k,p}(\Omega)}\Big ) ||h||_{W^{k,p}(\Omega)},$$ for some constant $C$ depending on $k, p,$ and $\Omega$. By Lemma \ref{lemma: composition bounded} above, taking $g_0$ to range over $\frac{\partial g}{\partial \zeta_i}$ and $\frac{\partial g}{\partial \overline{\zeta}_i}$, we obtain that the $W^{k,p}$ operator norm of $dg$ is uniformly controlled by that of $||u||_{W^{k,p}(\Omega)}$ and norms of derivatives of $g$. Claim (1) thus follows.

For claim (2), to bound $R_\beta(u,h),$ we need only a uniform bound on $(\partial^\beta g)\circ (u+th)(z)$ in terms of $u$ and $h$. We simply apply Lemma \ref{lemma: composition bounded} with $g_0=\partial^\beta g$, and then the claim is immediate.

As discussed above, the differentiability is now established. Finally, observe that $g$ is holomorphic if and only if $dg_u$ is $\C$-linear, which is equivalent to saying that $\Phi_g$ is Fréchet holomorphic.   
\end{proof}
Finally, we can prove Proposition \ref{prop: loc composition is holo}.
\begin{proof}[Proof of Proposition \ref{prop: loc composition is holo}]
    Choose a covering of $\Omega$ by relatively compact open subsets with Lipschitz boundary $\{\Omega_m\}_{m=1}^\infty$. The Fréchet structures on $W_{\textrm{loc}}^{k,p}(\Omega,U)$ and $W_{\textrm{loc}}^{k,p}(\Omega)$ can be generated by taking the $W^{k,p}$ norms over every $\Omega_m$. 
    
    Given $u\in W^{k,p}(\Omega,U)$, consider the restriction $u_m:=u|_{\Omega_m}$ in $W^{k,p}(\Omega_m,U).$ There exists a relatively compact subset $U_0$ of $U$ with Lipschitz boundary such that $u(\Omega_m)\subset U_0.$ Since the $L^\infty$ norm is continuous on $W^{k,p}(\Omega_m,U),$ for any sufficiently small variation $h\in W_0^{k,p}(\Omega)$ with restriction $h_m$ to $\Omega_m,$ we have $u_m+h_m\in W^{k,p}(\Omega_m,U_0)$. Using this observation, one can apply Lemma \ref{lemma: composition holomorphic} to see that the restriction $\Phi_g: W^{k,p}(\Omega_m,U)\to W^{k,p}(\Omega_m)$ is holomorphic. Hence, using our explicit choice of semi-norms for the Fréchet structure, one sees immediately that the map $\Phi_g: W_{\textrm{loc}}^{k,p}(\Omega,U)\to W_{\textrm{loc}}^{k,p}(\Omega)$ is holomorphic.
\end{proof}

\subsection{Proof of Theorem B}\label{sec: proof of theorem B}
Finally, using Theorem A and Proposition \ref{prop: loc composition is holo}, we prove Theorem B. Before we begin, we discuss the Fréchet manifold structure on $\mathcal{S}^l(S)$, which is made similarly to that of $\mathcal{C}^l(S,c_0).$ For now, let $l<\infty$. Recalling that $\mathcal{S}^l(S)$ is the space of $W_{\textrm{loc}}^{l,2}$ complex-valued symmetric bilinear forms on the complexified tangent bundle of $S$, we can more generally consider the space $\mathcal{S}^{l,p}(S)$ consisting of $W_{\textrm{loc}}^{l,p}$ bilinear forms (so, in our notation, $\mathcal{S}^{l,2}(S)=\mathcal{S}^{l}(S)$). Uniformizing $(\Tilde{S},c_0)$ to $\mathbb{H}$, $\mathcal{S}^{l,p}(S)$ becomes a closed subspace of the space $\mathcal{S}^{l,p}(\mathbb{H})$ of $W_{\textrm{loc}}^{l,p}$ complex symmetric bilinear forms on the complexified tangent bundle of $\mathbb{H}$. The closed subspace is defined by the condition that the bilinear forms satisfy the required transformation law in order to descend to $S$. For all $l\in \mathbb{Z}_{\geq 0}$ and $1<p<\infty,$ $\mathcal{S}^{l,p}(\mathbb{H})$ admits a Fréchet manifold structure, constructed using a relatively compact exhaustion analogous to the one made for $W_{\textrm{loc}}^{l,\infty}$ in Section \ref{subs: complex structures} and by taking $W_{\textrm{loc}}^{l,p}$ norms over the exhaustion. Then $\mathcal{S}^{l,p}(S)$ inherits this structure by inclusion. For $l=\infty$, we generate the Fréchet space structure by using all of the $W_{\textrm{loc}}^{l,p}$ semi-norms. We note here that the spaces $\mathcal{S}^{l,p}(\mathbb{H})$, $l\in \mathbb{Z}_{\geq 0}$, $1<p<\infty,$ will also appear in the proof of Theorem B. 

As well, we consider the space $W_{\textrm{loc}}^{l+1,p}(\mathbb{H}\cup \overline{\mathbb{H}})$. Explicitly, an element of $W_{\textrm{loc}}^{l+1,p}(\mathbb{H}\cup \overline{\mathbb{H}})$ is a function on $\mathbb{H}\cup \overline{\mathbb{H}}$ that restricts to a $W_{\textrm{loc}}^{l+1,p}$ function on $\mathbb{H}$ and on $\overline{\mathbb{H}}$. To define our semi-norms for the Fréchet space structure, we choose a covering by relatively compact open sets that are contained in either $\mathbb{H}$ or $\overline{\mathbb{H}}$. 
\begin{proof}[Proof of Theorem B]
For $\delta<1$, set $\mathcal{C}_\delta^l(S,c_0)$ to be the set of complex structures corresponding to $W_{\textrm{loc}}^{l,\infty}$ Beltrami forms with $L^\infty$ norm at most $\delta$, and define $\mathcal{C}_\delta^l(\overline{S},\overline{c_0})$ analogously. We will show that, for $p>2$ and $\eta=\frac{1}{N_p}$, the Bers metric map restricts to a holomorphic map $$\mathcal{C}_\eta^l(S,c_0)\times \mathcal{C}_\eta^l(\overline S,\overline{c_0})\to \mathcal{S}^{l,p}(S).$$ Since the inclusion of Fr{\'e}chet spaces $\mathcal{S}^{l,p}(S)\to \mathcal{S}^{l}(S)$ is holomorphic (more generally, $W_{\textrm{loc}}^{l,p_1}\to W_{\textrm{loc}}^{l,p_2}$ is holomorphic for $p_1>p_2)$, Theorem B will follow.

The starting point of the proof is that the Bers construction above defines a map from $\mathcal{C}_\eta^l(S,c_0)\times \mathcal{C}_\eta^l(\overline S,\overline{c_0})$ to $W_{\textrm{loc}}^{l+1,p}(\mathbb{H}\cup\overline{\mathbb{H}})$, by associating complex structures corresponding to Beltrami forms $\mu_1$ and $\mu_2$ to $f^{\hat{\mu}}|_{\mathbb{H}\cup\overline{\mathbb{H}}}.$ By Theorem A, this map is holomorphic. 

We consider $W_{\textrm{loc}}^{l+1,p}(\mathbb{H},\mathbb{C}\times\mathbb{C}\setminus \Delta),$ which is an open subset of the complex Fr{\'e}chet space $W_{\textrm{loc}}^{l+1,p}(\mathbb{H},\mathbb{C}\times\mathbb{C}),$ hence itself a complex Fr{\'e}chet manifold. We obtain a map $$(\f+, \fm)\colon \mathcal C^{l}_\eta(S,c_0)\times \mathcal{C}^{l}_\eta(\overline S,\overline{c_0})\to W_{\textrm{loc}}^{l+1,p}(\mathbb{H},\mathbb{C}\times\mathbb{C}\setminus \Delta)$$ by composing our map $\mathcal{C}_\eta^l(S,c_0)\times \mathcal{C}_\eta^l(\overline S,\overline{c_0})\to W_{\textrm{loc}}^{l+1,p}(\mathbb{H}\cup\overline{\mathbb{H}})$ with the linear holomorphic map $W_{\textrm{loc}}^{l+1,p}(\mathbb{H}\cup\overline{\mathbb{H}}) \to W_{\textrm{loc}}^{l+1,p}(\mathbb{H},\mathbb{C}\times\mathbb{C})$ that maps $
    \beta\in W_{\textrm{loc}}^{l+1,p}(\mathbb{H}\cup\overline{\mathbb{H}})$ to the function $z\mapsto (\beta|_{\mathbb H}(z), \beta|_{\overline{\mathbb H}}(\overline z))$.

To prove Theorem B, it suffices to prove that if $\inners$ is the holomorphic Riemannian metric on $\mathbb{CP}^1\times\mathbb{CP}^1\setminus \Delta$ from $\eqref{eq: holo metric}$ in the introduction, restricted to $\mathbb{C}\times\mathbb{C}\setminus \Delta,$ then the map $$W_{\textrm{loc}}^{l+1,p}(\mathbb{H},\mathbb{C}\times\mathbb{C}\setminus \Delta)\to \mathcal{S}^{l,p} (\mathbb{H}), \qquad  \hspace{1mm} u\mapsto u^*\inners,$$ is well-defined, i.e., lands in $\mathcal{S}^{l,p} (\mathbb{H})$, and is holomorphic. We point out that the argument that we use applies with minor change to any holomorphic metric on $\mathbb{C}\times\mathbb{C}\setminus \Delta.$

To this end, once we've established the result for all $l<\infty,$ the $l=\infty$ case follows, since, as mentioned in Section \ref{subs: complex structures}, we can generate the Fréchet structure on $C^\infty$ using the semi-norms coming from the $W_{\textrm{loc}}^{l,p}$ semi-norms. Hence, assume $l<\infty.$

If $(z_1,z_2)$ is the global coordinate on $\mathbb{C}\times \mathbb{C}$, then we can write $$\inners=\lambda(z_1,z_2)dz_1\cdot dz_2,$$ and hence if $u=(u_1,\overline{u_2}),$ $$u^*\inners = \lambda(u(z))du_1\cdot d\overline{u_2},$$ where $\inners$ is holomorphic. The map from $W_{\textrm{loc}}^{l+1,p}(\mathbb{H})$ to $W_{\textrm{loc}}^{l+1,p}$ complex forms on $\mathbb{H}$ given by $u\mapsto du$ is complex linear and continuous, hence holomorphic. By Proposition \ref{prop: loc composition is holo}, the map $W_{\textrm{loc}}^{l+1,p}(\mathbb{H}, \C\times\C\setminus \Delta)\to W_{\textrm{loc}}^{l+1,p}(\mathbb{H})$ given by $u\mapsto \lambda\circ u$ is holomorphic. We compose this map with the trivially holomorphic inclusion $W_{\textrm{loc}}^{l+1,p}(\mathbb{H})\to W_{\textrm{loc}}^{l,p}(\mathbb{H})$. Since $W_{\textrm{loc}}^{l,p}$ is a Fréchet algebra, we can apply the product rule to see that $u\mapsto u^*\inners$ is indeed holomorphic. Recalling our discussion above, the last assertion completes the proof of Theorem B.
\end{proof}

\printbibliography

@misc{ElE,
      title={A metric uniformization model for the Quasi-Fuchsian space}, 
      author={Christian El Emam},
      year={2023},
      eprint={2307.07388},
      archivePrefix={arXiv},
      primaryClass={math.DG}
}

@book{Ahl,
  title={Lectures on Quasiconformal Mappings},
  author={Ahlfors, L.V.},
  isbn={9780821882955},
  lccn={66006222},
  series={Mathematical studies},
  url={https://books.google.ca/books?id=aqhzAcdPrzIC},
  year={1966},
  publisher={Van Nostrand}
}

@book{chae1985holomorphy,
  title={Holomorphy and Calculus in Normed SPates},
  author={Chae, S.B.},
  isbn={9780824772314},
  lccn={84025985},
  series={Chapman \& Hall/CRC Pure and Applied Mathematics},
  url={https://books.google.lu/books?id=wtrvhCdbsd8C},
  year={1985},
  publisher={Taylor \& Francis}
}

@article {DZ,
    AUTHOR = {Dumitrescu, Sorin and Zeghib, Abdelghani},
     TITLE = {Global rigidity of holomorphic {R}iemannian metrics on compact
              complex 3-manifolds},
   JOURNAL = {Math. Ann.},
  FJOURNAL = {Mathematische Annalen},
    VOLUME = {345},
      YEAR = {2009},
    NUMBER = {1},
     PAGES = {53--81},
      ISSN = {0025-5831,1432-1807},
   MRCLASS = {53C56 (32Q57 53C24)},
  MRNUMBER = {2520052},
MRREVIEWER = {Bernd\ Kreussler},
       DOI = {10.1007/s00208-009-0342-8},
       URL = {https://doi.org/10.1007/s00208-009-0342-8},
}

@book {Ast,
    AUTHOR = {Astala, Kari and Iwaniec, Tadeusz and Martin, Gaven},
     TITLE = {Elliptic partial differential equations and quasiconformal
              mappings in the plane},
    SERIES = {Princeton Mathematical Series},
    VOLUME = {48},
 PUBLISHER = {Princeton University Press, Princeton, NJ},
      YEAR = {2009},
     PAGES = {xviii+677},
      ISBN = {978-0-691-13777-3},
   MRCLASS = {30C62 (30G20 35J46 35J60 35J92)},
  MRNUMBER = {2472875},
MRREVIEWER = {Olli\ Martio},
}

@article {BEE,
    AUTHOR = {Bonsante, Francesco and El Emam, Christian},
     TITLE = {On immersions of surfaces into {$SL(2,\Bbb C)$} and geometric
              consequences},
   JOURNAL = {Int. Math. Res. Not. IMRN},
  FJOURNAL = {International Mathematics Research Notices. IMRN},
      YEAR = {2022},
    NUMBER = {12},
     PAGES = {8803--8864},
      ISSN = {1073-7928,1687-0247},
   MRCLASS = {53C42 (32Q30 53C21 57K31)},
  MRNUMBER = {4436196},
MRREVIEWER = {Erasmo\ Caponio},
       DOI = {10.1093/imrn/rnab189},
       URL = {https://doi.org/10.1093/imrn/rnab189},
}

@article {AB,
    AUTHOR = {Ahlfors, Lars and Bers, Lipman},
     TITLE = {Riemann's mapping theorem for variable metrics},
   JOURNAL = {Ann. of Math. (2)},
  FJOURNAL = {Annals of Mathematics. Second Series},
    VOLUME = {72},
      YEAR = {1960},
     PAGES = {385--404},
      ISSN = {0003-486X},
   MRCLASS = {35.00 (30.00)},
  MRNUMBER = {115006},
MRREVIEWER = {C.\ B.\ Morrey, Jr.},
       DOI = {10.2307/1970141},
       URL = {https://doi.org/10.2307/1970141},
}

@article {SU,
    AUTHOR = {Bers, Lipman},
     TITLE = {Simultaneous uniformization},
   JOURNAL = {Bull. Amer. Math. Soc.},
  FJOURNAL = {Bulletin of the American Mathematical Society},
    VOLUME = {66},
      YEAR = {1960},
     PAGES = {94--97},
      ISSN = {0002-9904},
   MRCLASS = {30.00},
  MRNUMBER = {111834},
MRREVIEWER = {H.\ L.\ Royden},
       DOI = {10.1090/S0002-9904-1960-10413-2},
       URL = {https://doi.org/10.1090/S0002-9904-1960-10413-2},
}

@misc{RT,
      title={Complex Lagrangian minimal surfaces, bi-complex Higgs bundles and $\mathrm{SL}(3,\mathbb{C})$-quasi-Fuchsian representations}, 
      author={Nicholas Rungi and Andrea Tamburelli},
      year={2024},
      eprint={2406.14945},
      archivePrefix={arXiv},
      primaryClass={id='math.DG'}
}

@book {Sobolevbook,
    AUTHOR = {Runst, Thomas and Sickel, Winfried},
     TITLE = {Sobolev spaces of fractional order, {N}emytskij operators, and
              nonlinear partial differential equations},
    SERIES = {De Gruyter Series in Nonlinear Analysis and Applications},
    VOLUME = {3},
 PUBLISHER = {Walter de Gruyter \& Co., Berlin},
      YEAR = {1996},
     PAGES = {x+547},
      ISBN = {3-11-015113-8},
   MRCLASS = {47H30 (35J65 46E35 46F10 47N20)},
  MRNUMBER = {1419319},
MRREVIEWER = {P.\ Szeptycki},
       DOI = {10.1515/9783110812411},
       URL = {https://doi.org/10.1515/9783110812411},
}

@article {Kim,
    AUTHOR = {Kim, Young-Heon},
     TITLE = {Holomorphic extensions of {L}aplacians and their determinants},
   JOURNAL = {Adv. Math.},
  FJOURNAL = {Advances in Mathematics},
    VOLUME = {211},
      YEAR = {2007},
    NUMBER = {2},
     PAGES = {517--545},
      ISSN = {0001-8708,1090-2082},
   MRCLASS = {58J52 (32G15)},
  MRNUMBER = {2323536},
MRREVIEWER = {Yoonweon\ Lee},
       DOI = {10.1016/j.aim.2006.09.009},
       URL = {https://doi.org/10.1016/j.aim.2006.09.009},
}

@book{Dodson, 
place={Cambridge}, 
series={London Mathematical Society Lecture Note Series}, 
title={Geometry in a Fréchet Context: A Projective Limit Approach}, publisher={Cambridge University Press}, 
author={Dodson, C. T. J. and Galanis, George and Vassiliou, Efstathios}, 
year={2015}, 
collection={London Mathematical Society Lecture Note Series}}

@book {Ni,
    AUTHOR = {Nicolaescu, Liviu I.},
     TITLE = {Lectures on the geometry of manifolds},
   EDITION = {Second},
 PUBLISHER = {World Scientific Publishing Co. Pte. Ltd., Hackensack, NJ},
      YEAR = {2007},
     PAGES = {xviii+589},
       DOI = {10.1142/9789812770295},
    
}

@misc{ElSa,
      title={Complex affine spheres and a Bers theorem for SL(3,C)}, 
      author={Christian El Emam and Nathaniel Sagman},
      year={2025},
      eprint={2406.15287},
      archivePrefix={arXiv},
      primaryClass={math.DG},
      url={https://arxiv.org/abs/2406.15287}, 
}

@book {Hub,
    AUTHOR = {Hubbard, John Hamal},
     TITLE = {Teichm\"uller theory and applications to geometry, topology,
              and dynamics. {V}ol. 1},
      NOTE = {Teichm\"uller theory,
              With contributions by Adrien Douady, William Dunbar, Roland
              Roeder, Sylvain Bonnot, David Brown, Allen Hatcher, Chris
              Hruska and Sudeb Mitra,
              With forewords by William Thurston and Clifford Earle},
 PUBLISHER = {Matrix Editions, Ithaca, NY},
      YEAR = {2006},
     PAGES = {xx+459},
      ISBN = {978-0-9715766-2-9; 0-9715766-2-9},
   MRCLASS = {30F60 (30-02 30F10 32G15)},
  MRNUMBER = {2245223},
MRREVIEWER = {Hiroshige\ Shiga},
}

@misc{ElSa2,
      title={Complex harmonic maps and rank 2 higher Teichm\"uller theory}, 
      author={Christian El Emam and Nathaniel Sagman},
      year={2025},
      eprint={2506.11746},
      archivePrefix={arXiv},
      primaryClass={math.DG},
      url={https://arxiv.org/abs/2506.11746}, 
}

\end{document}